\documentclass{amsart}
\usepackage{amsthm,amstext,amsmath,amscd,amssymb,latexsym}
\usepackage{color}
\usepackage[matrix,arrow]{xy}
\usepackage{amsfonts,enumerate,array}
\usepackage[colorlinks=true]{hyperref}
\usepackage{todonotes}
	\usepackage{verbatim}
\usepackage{float}	
\usepackage{pgf,tikz}
\usepackage[toc,page]{appendix}

\newcommand{\PP}{{\mathbb P}}

\newcommand{\ZZ}{{\mathbb Z}}

\newcommand{\ls}{{\mathcal{L}}}
\newcommand{\Nn}{{\rm{N}}}

\DeclareMathOperator{\Pic}{Pic}

\DeclareMathOperator{\Eff}{Eff}

\DeclareMathOperator{\cone}{cone}

\def\Bl{\operatorname{Bl}}

\newcommand{\paper}{: \begin{it}}
\newcommand{\jour }{, \end{it}}

\newtheorem{theorem}{Theorem}[section]
\newtheorem{lemma}[theorem]{Lemma}

\newtheorem{corollary}[theorem]{Corollary}

\theoremstyle{definition}
\newtheorem{definition}[theorem]{Definition}

\theoremstyle{remark}
\newtheorem{remark}[theorem]{Remark}

\numberwithin{equation}{section}

\begin{document}

\title{Cones of effective divisors on the blown-up $\PP^3$ in general lines}

\author{Olivia Dumitrescu}
\email{{\tt  dumitrescu@mpim-bonn.mpg.de }}
\address{Max-Planck Institute for Mathematics, Vivatsgasse 7, 53111 Bonn, Germany}

\author{Elisa Postinghel}
\email{{\tt elisa.postinghel@wis.kuleuven.be}}
\address{KU Leuven, Department of Mathematics, Celestijnenlaan 200B, 3001 Heverlee,
Belgium}

\author{Stefano Urbinati}
\email{{\tt urbinati@math.unipd.it}}
\address{Universita' degli Studi di Padova, Dipartimento di Matematica, Via Trieste 63, 35121 Padova, Italy}

\thanks{The first author is a member of 
the Simion Stoilow Institute of Mathematics of the 
Romanian Academy. The second author is supported by the Research Foundation - Flanders (FWO). The third author is supported by the PISCOPIA cofund Marie Curie Fellowship Programme.}

\vspace{-0.8cm}

\begin{abstract}
We compute the facets of the effective cones of divisors on the blow-up of $\PP^3$ in up to
 five lines in general position. We prove that up to six lines these threefolds are weak Fano and hence Mori Dream Spaces.
\end{abstract}

\maketitle

\section{Introduction}

In classical algebraic geometry, the study of linear systems in $\PP^n$ 
of hypersurfaces of degree $d$ with prescribed multiplicities at a collection of $s$ points in
 general position was investigated by many authors for over a century
(see \cite{Ciliberto} for an overview). 

We will briefly recall the most important results in this area. The well-known 
Alexander-Hirschowitz  theorem \cite{AlHi} classifies completely the dimensionality problem
 for linear systems with double points (see \cite{ale-hirsch, Ch1, Ch2, Po} for more recent and simplified 
proofs). Besides this theorem, general results about linear systems are rare and few things are known. 
For arbitrary number of points in the projective space  $\PP^n$, the dimensionality problem for linear systems
was analysed in the articles \cite{BDP1,BDP3,DumPos}. 
On the one hand, linear systems can be studied via techniques of commutative algebra. 
In fact, the dimensionality problem is related via apolarity to the Fr\"oberg-Iarrobino 
Weak and Strong conjectures \cite{Froberg, Iarrobino}. 
These conjectures give a predicted value for the Hilbert series of an  ideal generated by 
$s$ general powers of linear forms  in the polynomial ring with $n+1$ variables. 
On the other hand, linear systems correspond in birational geometry to the space 
of global sections of line bundles on the blown-up space $\PP^n$ in points. 
Form this perspective, the dimensionality problem of such linear systems reduces to vanishing
 theorems of divisors on blown-up spaces in points and along higher dimensional
cycles of the base locus. Vanishing theorems are important in algebraic geometry because they are related to positivity properties of divisors. In this direction, there are 
positivity conjectures in birational geometry that involve complete understanding of cohomology groups 
of divisors (see \cite{DP} for an introduction). 

 For a small number of points, namely whenever $s\leq n+3$, understanding the cohomology of
 divisors on the blown-up $\PP^n$ in $s$ points is easier since this space is a Mori dream space, see \cite[Theorem 1.3]{CT}.
A projective variety $X$ is called a Mori dream space if its Cox ring is finitely generated. The Cox ring of an algebraic variety naturally generalizes the homogeneous coordinate ring of the projective space.
 Mori dream spaces are generalizations of toric varieties that have a polynomial Cox ring and their geometry can be encoded by combinatorial data. For example, Mori dream spaces have rational polyhedral effective cone.
 In other words, the cone of effective divisors can be described as intersection of
 half-spaces in  $\Nn^1(X)_\mathbb{R}\cong\mathbb{R}^m$, the N\'eron-Severi
 group of $X$ tensored with the real numbers.

Motivated by the study of the classical interpolation problems in $\PP^n$ started by the Italian 
school of algebraic geometry,
 in this article we study the blown-up projective space $\PP^3$ in a collection of $s$ lines in
 general position $l_1, \ldots, l_s$. 
We give a description of the cone of effective divisors whenever $s$ is bounded above by five. 
The technique for finding the facets of the effective cone was first developed for
 blown-up $\PP^n$ in $n+3$ points in \cite{BDP3}.
 
Moreover, in this article we prove that these threefolds are weak Fano, hence Mori dream spaces, for
a number of lines bounded above by six. We also prove that for $s \geq 7$ these threefolds are not weak Fano.

These notes are organized as follows. In Section \ref{preliminary} we introduce the 
notation that will be used throughout.
In Section \ref{base locus} we describe cycles of the base locus of linear systems of
divisors on the blown-up  space in a collection of lines in general position. 
In section \ref{section:weakfano} we prove the first main result of these notes, i.e. we bound the number of blown-up general lines to obtain a weak Fano threefolds.
Section  \ref{effective cones}  contains the second main result of these notes, a complete description of the effective cone of the blown-up  projective space  at up to five general lines.

\subsection*{Acknowledgments} 
The authors would like to thank the organizers of the workshop 
``Recent advances in Linear series and Newton-Okounkov bodies'', Universit\`a degli Studi di Padova, 2015,
for their hospitality and financial support. 
This collaboration and project was started there. 
We would also like to thank M. Bolognesi and M. C. Brambilla for helpful conversations.
We thank the referee for useful comments and suggestions on the first version of this manuscript.

\section{Preliminaries and notation}\label{preliminary}
Let $X_s:=\Bl_s\PP^3$ be the blow-up of $\PP^3$ along $s$ lines $l_1,\dots, l_s$ 
in general position 
and let $E_1,\ldots,E_s$ denote the corresponding exceptional divisors. 
 We use $H$ to denote the class of the pull-back of a general hyperplane in $\Bl_s\PP^3$.
 The Picard group
$\Pic (\Bl_s \PP^3)$ is spanned by the hyperplane class $H$ and the exceptional divisors $E_i$. 

\begin{remark}\label{hyperplane}
Notice that the hyperplane class $H$ is a blown-up projective plane at $s$ points, 
with exceptional divisors $e_i=E_i|_H$, $i=1,\dots,s$, and line class $h:=H|_H$. 
Each exceptional divisor $E_i$ is a Hirzebruch surface ${\mathbb F}_0$, isomorphic
 to  $\PP^1\times \PP^1$, and the restriction of the hyperplane class
 $H|_{E_i}$ is the class of a fiber.
\end{remark}

Let 
\begin{equation}\label{divisor}
D=dH-\sum_{i=1}^sm_iE_i,
\end{equation}
 with $d,m_1,\dots,m_s\in\ZZ$, be any line bundle  in $\Pic(X_s)$.

\begin{remark}\label{l.s. in P^3}
If $d,m_i\ge0$, $|D|$ corresponds to the linear system $\ls_{d}(m_1,\dots,m_s)$ of degree-$d$ surfaces of $\PP^3$ 
that have multiplicity at least $m_i$ along $l_i$, for $i\in\{1,\dots,s\}$.
\end{remark}

\begin{remark}\label{rmk:log fano}
The canonical divisor of the blown-up projective space $\PP^3$ in $s$ lines is
\begin{equation}\label{canonical}
K_{\Bl_s\PP^3}=-4H+\sum_{i=1}^sE_i.
\end{equation}
\end{remark}

\section{Base locus lemmas}\label{base locus}

\begin{lemma}\label{base locus for quadric}
Let $\ls_d(m_1,\ldots,m_s)$ be the linear system  interpolating $s$ lines with multiplicities $m_i$. If $l_i,l_j,l_k\subset\PP^3$, with $i, j, k \in \{1, \ldots, s\}$, are three lines in general position, then  there exists a (unique) smooth quadric surfaces $Q_{ijk}$ in $\PP^3$ containing the lines $l_i,l_j,l_k$. Moreover such a quadric splits off $\ls_d(m_1,\ldots,m_s)$ at least $k_{ijk}:=\max(m_i+m_j+m_p-d,0)$ times.

\end{lemma}

\begin{proof} It is enough to prove the claim for $s=3$, and we will denote the lines by $l_1, l_2 , l_3$.
One can easily compute that $h^0(\PP^3,\ls_{2}(1,1,1))=1$; the quadric is the unique 
element of this linear system and it is isomorphic to 
$\PP^1\times\PP^1$.
By generality, $l_1,l_2,l_3\subset\PP^3$ belong to the same ruling of $Q_{123}$. Notice that each element of $\ls_d(m_1,m_2,m_3)$ intersects each line of the other ruling of the quadric at least $m_1+m_2+m_3$ times.
Therefore, if $m_1+m_2+m_3>d$, the quadric $Q_{123}$ is in the base locus of $\ls_d(m_1,m_2,m_3).$ 
To compute the multiplicity of containment one can check that in the residual linear system, $\ls_d(m_1,m_2,m_3)-Q_{123}$, the integer $k_{123}$ drops by $1$
and this concludes the proof. 
\end{proof}

\begin{lemma}\label{base locus for transversals}

Let $\ls_d(m_1,\ldots, m_s)$ as in Lemma \ref{base locus for quadric}. 
Any collection of four general lines $l_i, l_j, l_k, l_{\iota}$ in $\PP^3$ determines  two \emph{transversal} lines $t,t'$. For any $s\geq 4$ these transversals are contained in the base locus of the linear system $\ls_d(m_1,\ldots, m_s)$ with multiplicity at least
 $k_{t}=k_{t'}:=\max(m_i+m_j+m_k+m_{\iota}-d,0)$.
\end{lemma}
\begin{proof}

It is enough to prove the claim for $s=4$, and for simplicity we will denote the lines by $l_1,l_2,l_3,l_4$. 
The existence of two distinct transversals lines is straightforward. We  will denote by  $t,t'$ such lines, see Figure \ref{fig:Q123}.

Notice that all elements of 
$\ls_d(m_1,m_2,m_3,m_4)$ intersect $t$ (resp. $t'$) 
at least $m_1+m_2+m_3+m_4$ times, therefore if $m_1+m_2+m_3+m_4>d$ the two transversals are part of the base locus of $\ls_d(m_1,m_2,m_3,m_4)$. In fact, the multiplicity of containment of these transversals is at least $k_{t}$.

\end{proof}

\begin{figure}[H]
\centering

 \begin{tikzpicture}[line cap=round,line join=round,x=1.0cm,y=1.0cm]
\node at (4.5,3.8) {$Q_{123}$};
\draw[color=gray, thick] (0,0) to[out=5,in=175] (4,0);
\draw[color=gray, thick]  (4, 0) to[out=100,in=260](4,4);
\draw[color=gray, thick]  (4,4)to[out=185,in=-5] (0,4);
\draw[color=gray, thick]  (0,4) to[out=-80,in=80] (0,0);

\draw[color=red, thick] (.5, 2.5)--(3.5, .5);
\draw[color=red, thick] (.5, 3)--(3.5, 1);
\draw[color=red, thick] (.5, 3.5)--(3.5, 1.5);

\draw[color=black!50!green, thick] (3.5,3.5)--(.5, 1);
\draw[color=black!50!green, thick] (3.5, 2.7)--(1, .5);

\draw (1,1.45) node[circle,fill,inner sep=1.5pt,label=below:$p_4$, color=red]{};
\draw (1.5,.95) node[circle,fill,inner sep=1.5pt,label=below:$p'_4$, color=red]{};

\node at (0.35,3.5) {$l_1$};
\node at (0.35,3) {$l_2$};
\node at (0.35,2.5) {$l_3$};

\node at (3.65,3.5) {$t$};
\node at (3.65,2.7) {$t'$};

\end{tikzpicture}
\caption{$Q_{123}$} \label{fig:Q123}
\end{figure}

\section{Weak Fano case} \label{section:weakfano}
In this section we prove that  $X_s=\Bl_s\PP^3$, the blow-up of $\PP^3$ in $s$ lines in general position, is \emph{weak Fano} if and only if $s\le 6$. In particular this implies that for $s\le 6$ it is a Mori dream space. The notion of weak Fano  varieties relates to positivity properties of linear series.

Notice that knowing the base locus of the linear series associated to a given divisor $D$ makes it simpler to determine when D is nef, 
since the curves that intersect the divisor negatively can only be contained in the diminished base locus of the divisor $\mathbb{B}_-(D)$ (see \cite{ELMNP}).

\medskip

Let us first recall some definitions and properties connected to the Minimal Model Program that give a precise way of detecting Mori dream spaces.
\begin{definition}\label{weak-log}
Let $X$  be a smooth projective variety. We say that
\begin{itemize}
\item $X$ is {\it weak Fano} if $-K_X$ is big and nef.
\item $X$ is {\it log Fano} if there exists an effective divisor $D$ such that $-(K_X +D)$ is ample and the pair $(X,D)$ has {\it Kawamata log terminal} singularities (see \cite{BCHM} for the definition).
\end{itemize}
\end{definition}

\begin{remark}\label{weak-log-mds}
We have $$\text{ weak Fano} \Rightarrow \text{log Fano} \Rightarrow \text{Mori Dream Space}.$$
The first implication follows trivially from Definition \ref{weak-log}, while the second one follows from the results contained in \cite{BCHM}.
\end{remark}

 If $s\le 2$, the threefold $X_s$ is a normal toric variety. 
 For $s=3$,  $X_s$ is isomorphic to the blow-up of $\PP^1\times\PP^1\times\PP^1$ along the small diagonal. This variety is one of the Fano 3-folds
of Picard number four in Fujita's classification, see \cite{MM}. 
Therefore $X_3$ is  log Fano, hence a Mori dream space. 
For $s\ge 4$  it was not known, at the best of our knowledge, 
whether $X_s$  is  a Mori dream space.

\begin{remark} \label{weakfano}
The intersection table on $X_s$ is given by the following:
$$
H^3=1,\quad H^2E_i=0,\quad H E_i^2=-1, \quad E_i^2E_j=0, \quad E_i^3=2,
$$
for every distinct indices $i,j\in\{1,\dots,s\}$.
Therefore we can compute the top self-intersection of the anti-canonical divisor: $(-K_{X_s})^3 = 64 - 10s$.  If $-K_{X_s}$ is nef, it  fails to be big whenever $s \geq 7$, see e.g. \cite[Theorem 2.2.16]{LazarsfeldI}.
\end{remark}

\begin{theorem}\label{weak fano} The threefold $X_s$ is weak Fano if and only if $s\leq 6.$
\end{theorem}

\begin{proof}
By Remark \ref{weakfano}, it is enough to prove that $-K_s$ is nef for $s \le 6$.
For $s\le 4$, $-K_s$ is clearly nef, because it can be written as a sum of nef divisors given by pencils $H-E_i$. 

We will consider the cases $s=5,6$.
If $s=5$,  notice that for every permutation $\sigma\in\mathfrak{S}_5$, the union of divisors $$(2H-E_{\sigma(1)}-E_{\sigma(2)}-E_{\sigma(3)})+(H-E_{\sigma(4)})+(H-E_{\sigma(5)})$$ belongs to the anti-canonical system. 
Similarly, if $s=6$, for every permutation $\sigma\in\mathfrak{S}_6$, the union $$(2H-E_{\sigma(1)}-E_{\sigma(2)}-E_{\sigma(3)})+(2H-E_{\sigma(4)}-E_{\sigma(5)}-E_{\sigma(6)})$$ belongs to the anti-canonical system. Therefore the base locus of $-K_{X_s}$, for $s=5,6$,  is contained in the intersection of all these unions, which is $\sum_{i=1}^sE_i$.
We also know that any curve $C$ intersecting $-K_X$ negatively is
contained in $\mathbb{B}_-(-K_X) \subseteq \bigcup_{i=1}^sE_i$.
Now, since for any $i\in\{1,\dots,s\}$ the exceptional divisor $E_i$ is isomorphic to $\PP^1 \times \PP^1$, any curve $C$ contained in $E_i$ is such that 
 $C\cdot E_i \leq 0$  and $C\cdot E_j=0$ for any $j\in\{1,\dots,s\}\setminus\{i\}$. 
In particular this implies that $-K_{X_s}$ is nef.
\end{proof}

The following is an immediate consequence of Theorem \ref{weak-log-mds} and Remark \ref{weak-log-mds}
\begin{corollary}
If $s\le 6$, the threefold $X_s$ is a Mori dream space.
\end{corollary}

\section{Cones of effective divisors}\label{effective cones}

In this section we will consider $\mathbb{Q}$-Cartier divisors  on $X_s=\Bl_s\PP^3$, the blow-up
of $\PP^3$ in $s$ general lines. Any divisor here will have the form \eqref{divisor}, where the assumption on
the coefficients has been relaxed to $d,m_i\in \mathbb{Q}$. The purpose is in fact to 
describe the boundary of the effective cone.
 These cones can be extremely complicated for a general variety.
 In the specific case of Mori dream spaces they are  polyhedral but there is no reason to expect that the hyperplanes 
cutting the facets of such cone are described by integral equations
in the degree and the multiplicities. 

 For  $s\le 2$, since $X_s$ is a normal toric variety, 
 the cone of effective divisors can be computed by means of techniques from toric geometry, see for instance \cite{CLS}.
In this section we extend the description of the effective cone
from the toric case to the case $s=3,4,5$, in particular showing that it is a rational and 
polyhedral cone.

\medskip

We will let $(d, m_1, \dots, m_s)$  be the coordinates of the
N\'eron-Severi group of $X_s$. In the next sections we will give equations describing
the cones of effective divisors for $s\le 5$.

\subsection{The case of three lines}

\begin{theorem}\label{facets eff 3}
If $s\le3$, the effective cone of $X_3$ is the closed rational polyhedral cone given by the
 following inequalities:
\begin{align}
0&\le d,\label{1.3} \\ 
m_i&\le d, \ \forall i\in\{1,\dots,s\}, \label{2.3}\\
 m_i+m_j&\le d, \ \forall i,j\in\{1,\dots,s\}, i\ne j \ (\textrm{if } s=2,3).\label{3.3}
\end{align}
\end{theorem}

\begin{proof}
Notice that if $m_i<0$ for some $i\in\{1,\dots,s\}$, then $D$ is effective if and only if $D+m_iE_i$ is. 
Hence we may assume that $m_i\ge0$ for all $i=1,\dots,s$. 
In this case \eqref{3.3} implies \eqref{2.3}.  

Assume first of all that $D$ is effective. Clearly $d\ge0$.
Assume $m_i+m_j>d$, for some $i\ne j$. Each point $p$ of $\PP^3 \smallsetminus \{l_1,\dots,l_s\}$ 
lies on the line spanned by points $p_i\in l_i$ and  $p_j\in l_j$.
 By B{\'e}zout's Theorem such a line  is contained in the base locus of $\ls_d(m_1,\dots,m_s)$. 
Hence  by Remark \ref{l.s. in P^3}, the strict transform on $X_s$ of such a line 
 is contained in the base locus of $|D|$, and in particular so does 
the inverse image of $p$. Therefore each point of 
$X_s\smallsetminus\{E_1,\dots,E_s\}$ is in the
base locus of $|D|$, hence $|D|=\emptyset$. 

\medskip 

We now prove that a divisor $D$ that satisfies the inequalities is effective.

If $s=2$, we can write $D=(d-m_1-m_2)H+m_1(H-E_1)+m_2(H-E_2)$ and conclude, 
as $D$ is sum of effective divisors with non-negative coefficients.

Assume $s=3$ and $m_1+m_2+m_3-d\le0$. 
We can for instance write $D=(d-m_1-m_2-m_3)E_1+(d-m_2-m_3)(H-E_1)+m_2(H-E_2)+m_3(H-E_3)$ and conclude.
If, instead, $k_{123}=m_1+m_2+m_3-d>0$, we can write $D=k_{123}(2H-E_1-E_2-E_3)+(d-m_2-m_3)(H-E_1)+(d-m_1-m_3)(H-E_2)+(d-m_1-m_2)(H-E_3)$. 
\end{proof}

\subsection{The case of four lines}

\begin{theorem}\label{facets eff 4}
The effective cone of $X_4$ is the closed rational polyhedral cone given by the following inequalities:

\begin{align}
0&\le d,  \label{1.4} \\ 
m_i&\le d,   \ \forall i\in\{1,\dots,4\}, \label{2.4} \\
 m_i+m_j&\le d, \ \forall i,j\in\{1,\dots,4\}, i\ne j, \label{3.4}\\
(m_1+\cdots+m_4)+m_i&\leq 2d,  \ \forall i\in\{1,\dots,4\}, \label{5.4}\\
2(m_1+\cdots+m_4)&\le 3d \label{4.4}.
\end{align}
\end{theorem}

\begin{proof}
As in the proof of Theorem \ref{facets eff 3}, we may assume that $m_i\ge0$, 
for all $i=1,2,3,4$. In this case \eqref{2.4} is redundant.

If $D$ is effective, then \eqref{1.4} and \eqref{3.4} are satisfied by Theorem  
\ref{facets eff 3}. 

We now prove that $D\ge0$ implies  \eqref{5.4}.
Notice that, without loss of generality, we may prove that the inequality is
 satisfied for a choice of index $i\in\{1,2,3,4\}$. 
Set $i=1$ and consider the pencil of planes  containing $l_1$. 
Notice that it covers the whole space $\PP^3$. Therefore in order 
to conclude it is enough to prove that if the inequality is violated, 
then each point of the general plane
  of the pencil is contained in the
 base locus of $|D|$, making it empty. 
Let $\Delta$ be a general plane containing $l_1$. Each line $l_j$, $j=2,3,4$, 
intersects $\Delta$ in 
a point: 
write $p_j:=l_i\cap\Delta$ (Figure \ref{fig:l_1}).  By the generality assumption,
the  points $p_2,p_3,p_4$ are not collinear. We will
consider the net of conics in $\Delta$ passing through these
three points.
\begin{figure}[H]
\centering
\begin{tikzpicture}[line cap=round,line join=round,x=1.0cm,y=1.0cm]
\node at (5.3,3.8) {$\Delta$};
\draw[color=gray, thick] (0,0) -- (4,0);
\draw[color=gray, thick]  (4, 0) -- (5,4);
\draw[color=gray, thick]  (5,4)-- (1,4);
\draw[color=gray, thick]  (1,4) --(0,0);
\draw[color=red, thick] (1.5, 3.5)--(3.5, .5);
\node at (1.35,3.5) {$l_1$};
\draw (1,2) node[circle,fill,inner sep=1.5pt,label=below:$p_4$, color=red]{};
\draw (1.5,1) node[circle,fill,inner sep=1.5pt,label=below:$p_2$, color=red]{};
\draw (3,3) node[circle,fill,inner sep=1.5pt,label=below:$p_3$, color=red]{};
\end{tikzpicture}
\caption{$\Delta$} \label{fig:l_1}
\end{figure}
Let us denote by $\tilde{\Delta}$
the  strict transform of $\Delta$ in $X_4$, that 
 is a projective plane blown-up in three non-collinear points. 
The exceptional divisor $E_1$ intersects the blown-up plane $\tilde{\Delta}$
in the class of a general line.
Precisely, we have the following intersection table:
\begin{align*}
h:=& H|_{\tilde{\Delta}}=E_1|_{\tilde{\Delta}}, \\
e_j:=& E_j|_{\tilde{\Delta}}, \quad j=2,3,4.
\end{align*}
The classes $h$ and $e_j$ for
$j=2,3,4$ generate the Picard group of ${\tilde{\Delta}}$. 
Consider the restriction 
$$
D|_{\tilde{\Delta}}=
 (d-m_1)h-\sum_{j=2}^4m_je_j.
$$
The intersection number between $D|_{\tilde{\Delta}}$ and the strict transform on 
$\tilde{\Delta}$ of a conic through $p_2,p_3,p_4$
is $2(d-m_1)-m_2-m_3-m_4$. 
If this number is negative, namely if \eqref{5.4} is violated with $i=1$, then each conic through $p_2, p_3, p_4$
 is contained in the base locus of $|D|_{\tilde{\Delta}}|$. Since these conics 
cover $\Delta$, we conclude $\tilde{\Delta}$ is contained in the base locus of $|D|$. This is a contradiction since a pencil of planes can not be contained in the fixed part of an effective divisor.

We now prove  the last inequality.
Notice first of all that if $m_1+m_2+m_3+m_4\le d$, then \eqref{4.4} 
follows obviously from \eqref{3.4}. Hence we may assume that
$m_1+m_2+m_3+m_4>d$.
Observe that each point of $X_4\smallsetminus\{E_1,\dots,E_4\}$ 
sits on the strict transform of a line in $\PP^3$ that is transversal to $t$ and $t'$,
because $t$ and $t'$ are skew.
Using Lemma \ref{base locus for transversals} and arguing as for the proof of Theorem \ref{facets eff 3}, we conclude that if $2(m_1+\cdots+m_4)> 3d$ then $|D|$ is empty.

\medskip 

We now prove that a divisor $D$ satisfying \eqref{1.4}, \eqref{3.4}, \eqref{5.4} and \eqref{4.4} is effective.

Consider first the case $m_{i_1}+m_{i_2}+m_{i_3}\le d$ for all $\{i_1,i_2,i_3\}\subset\{1,2,3,4\}$. Write $D=D'+m_4(H-E_4)$. Then
$$D'=(d-m_4)H-m_1E_1-m_2E_2-m_3E_3$$ is effective by Theorem  \ref{facets eff 3}.
Hence we conclude in this case.

After reordering the  indices $\{1,2,3,4\}$ if necessary, we assume that
$k_{123}=m_1+m_2+m_3-d>0$. It is enough to prove that 
$$
D'':=D-k_{123}(2H-E_1-E_2-E_3)
$$ 
satisfies \eqref{1.4},  \eqref{3.4}, \eqref{5.4} and \eqref{4.4} to conclude. 
The coefficients of the expanded expression 
$$
D''=:d'H-\sum_{i=1}^3m'_iE_i-m_4E_4
$$ 

are $d'=3d-2(m_1+m_2+m_3)$, $m'_1=d-m_2-m_3$, $m'_2=d-m_1-m_3$, $m'_3=d-m_1-m_2$. 
Since all of these coefficients are positive, it is enough to check that \eqref{3.4}, \eqref{5.4}
 and \eqref{4.4} are satisfied. 
The divisor $D''$  satisfies condition  \eqref{3.4} because $D$ does. 
Indeed $m'_i+m'_j\le d'$ is equivalent to $m_i+m_j\le d$,  for $1\le i<j\le3$; moreover 
$m'_i+m_4\le d'$  is equivalent to $(m_1+m_2+m_3+m_4)+m_i\le 2d$, for all $1\le i\le 3$, i.e. condition \eqref{5.4}. 
Furthermore $D''$ satisfies \eqref{5.4} with $i\in\{1,2,3\}$ if an only if $D$ does; 
$D''$ satisfies  \eqref{5.4} with $i=4$ if and only if $D$ satisfies \eqref{4.4}.
Finally, $D''$ satisfies \eqref{4.4} because $D$ does; we leave the details to the reader. 
\end{proof}

%
%

\subsection{The case of five lines}

\begin{lemma}\label{cubic surface}
There does not  exist any cubic surface containing five general lines of $\mathbb{P}^3$.
\end{lemma}
\begin{proof}
Assume that there exists a cubic surface $S$ containing $5$ general lines of $\PP^3$. 
It has to be irreducible. In fact  there is no quadric surface containing four lines and no plane containing two lines. 
Therefore $S|_{\Delta}$ is a plane cubic, for each plane $\Delta\subset\PP^3$.

Now, for every $i=1,\dots,5$, denote by  $t_i,t'_i$  the two transversal lines to the four lines   $\{l_1,\dots,\check{l}_i,\dots,l_5\}\subset\PP^3$.
By Lemma \ref{base locus for transversals}, $t_i,t'_i$ are contained in $S$, 
for every $i=1,\dots,5$.  
Let now $\Delta\subset\PP^3$ be the plane spanned by $l_1$ and $t_5$ (see Figure \ref{fig:l_1t_5}).

\begin{figure}[H]
\centering
\begin{tikzpicture}[line cap=round,line join=round,x=1.0cm,y=1.0cm]
\node at (5.3,3.8) {$\Delta$};
\draw[color=gray, thick] (0,0) -- (4,0);
\draw[color=gray, thick]  (4, 0) -- (5,4);
\draw[color=gray, thick]  (5,4)-- (1,4);
\draw[color=gray, thick]  (1,4) --(0,0);

\draw[color=red, thick] (1.5, 3.5)--(3.5, .5);
\node at (1.35,3.5) {$l_1$};

\draw[color=black!50!green, thick] (4.5, 3.5)--(1,1);
\node at (4.5,3.1) {$t_5$};

\draw (1.5,2.8) node[circle,fill,inner sep=1.5pt,label=below:$p_5$, color=red]{};
\draw (3,3.5) node[circle,fill,inner sep=1.5pt,label=below:$q_1$, color=black!50!green]{};
\draw (3.8,1.5) node[circle,fill,inner sep=1.5pt,label=below:$q_1'$, color=black!50!green]{};

\draw (1.55,1.4) node[circle,fill,inner sep=1.5pt,label=below:$p_4$, color=red]{};
\draw (2.2,1.85) node[circle,fill,inner sep=1.5pt,label=below:$p_3$, color=red]{};
\draw (3.3,2.65) node[circle,fill,inner sep=1.5pt,label=below:$p_2$, color=red]{};

\end{tikzpicture}

\caption{$\Delta_1^5$} \label{fig:l_1t_5}
\end{figure}
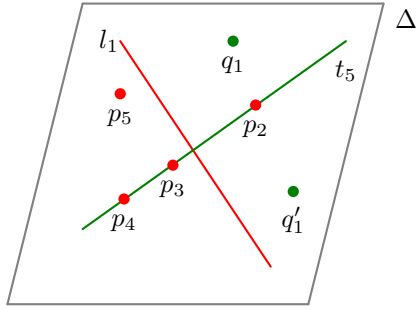

The restricted linear series $|S|_\Delta|$  is  the linear series of plane cubics containing $l_1$ and $t_5$ and passing through the three points $l_5\cap\Delta, t_1\cap\Delta, t'_1\cap\Delta$. By the generality assumption, these three points are not collinear, hence the restricted  series is empty. This concludes the proof. 
\end{proof}

 \begin{remark}
Assuming that a collection of lines is in general position, is  much stronger than 
assuming that the lines are skew. 

Recall that  any smooth cubic surface $S$ may be realized as the blow-up of the projective plane $\PP^2$ in six points $p_1,\dots,p_6$ in general position. There are $27$ lines in $S$:  the $6$ exceptional divisors, the strict transforms of the lines through two of the $p_i$'s, and the strict transforms of the conics through five of the $p_i$'s.

Therefore, for any choice of five among the exceptional divisors on the blown-up $\PP^2$, there is a
line intersecting them, whereas for general lines on the cubic surface, this can happen for at most four lines, see Lemma \ref{base locus for transversals}. 
\end{remark}

\begin{theorem}\label{facets eff 5}
The effective cone of $X_5$ is the closed rational polyhedral cone given by the following inequalities:
\begin{align}
0&\le d,  \label{1.5} \\ 
m_i&\le d,  \ \forall i\in\{1,\dots,5\}, \label{2.5} \\
 m_i+m_j&\le d, \ \forall i,j\in\{1,\dots,5\}, i\ne j, \label{3.5}\\
(m_1+\cdots+m_5)+m_i-m_j&\le 2d, \ \forall i,j\in\{1,\dots,5\}, i\ne j, \label{4.5}\\
(m_1+\cdots+m_5)+m_i&\le 2d, \ \forall i\in\{1,\dots,5\}, \label{5.5}\\
2(m_1+\cdots+m_5)&\le 3d.\label{6.5}
\end{align}
\end{theorem}

\begin{proof}
We may assume that $m_i\ge0$, 
for all $i=1,2,3,4,5$. In this case \eqref{2.5} and \eqref{4.5} are redundant. 

We prove that the conditions \eqref{1.5}$,\dots,$\eqref{6.5} are necessary for a divisor $D$ on $\Bl_5\PP^3$ 
to be effective. 

If $D$ is effective, then \eqref{1.5} and \eqref{3.5} are satisfied by Theorem  
\ref{facets eff 3}.

The proof that $D\ge0$ implies  \eqref{5.5} is the same as the one for the inequality 
\eqref{5.4} in Theorem \ref{facets eff 4}, where we consider the pencil of conics in $\Delta$ 
passing 
to four points $p_2,p_3,p_4$ and $p_5:=l_5\cap\Delta$.

We now prove that $D\ge0$ implies \eqref{6.5}. 
If $m_i=0$ for some $i\in\{1,\dots,5\}$, we conclude by Theorem \ref{facets eff 4}; therefore we may assume $m_i>0$.
If $(m_1+\cdots+m_5)-m_i\le d$ for some $i\in\{1,\dots,5\}$, then
\eqref{6.5}  follows from \eqref{5.5}. 
Therefore we shall assume $(m_1+\cdots+m_5)-m_i> d$ for every $i$.
Assume by contradiction that \eqref{6.5} is not satisfied.
Let us consider the plane $\Delta$ containing $l_1$ (Figure \ref{fig:l_1_1}). 
\begin{figure}[H]
\centering
\begin{tikzpicture}[line cap=round,line join=round,x=1.0cm,y=1.0cm]
\node at (5.3,3.8) {$\Delta$};
\draw[color=gray, thick] (0,0) -- (4,0);
\draw[color=gray, thick]  (4, 0) -- (5,4);
\draw[color=gray, thick]  (5,4)-- (1,4);
\draw[color=gray, thick]  (1,4) --(0,0);

\draw[color=red, thick] (1.5, 3.5)--(3.5, .5);
\node at (1.35,3.5) {$l_1$};

\draw (1,2) node[circle,fill,inner sep=1.5pt,label=below:$p_4$, color=red]{};
\draw (1.5,1) node[circle,fill,inner sep=1.5pt,label=below:$p_2$, color=red]{};
\draw (3,3) node[circle,fill,inner sep=1.5pt,label=below:$p_3$, color=red]{};
\draw (2.3,3.5) node[circle,fill,inner sep=1.5pt,label=below:$p_5$, color=red]{};
\draw (4,3.5) node[circle,fill,inner sep=1.5pt,label=below:$q_1$, color=black!50!green]{};
\draw (3.5,1.5) node[circle,fill,inner sep=1.5pt,label=below:$q_1'$, color=black!50!green]{};
\end{tikzpicture}
\caption{$\Delta'_1$} \label{fig:l_1_1}
\end{figure}
Let $\Delta$, $t_i$ and $t'_i$ be as above and set
 $q_1:=t_1\cap\Delta$ and $q'_1:=t'_1\cap\Delta$.  
Assume $\Delta$ is general. Note that the set of points
$\{q_1,p_2,\dots,p_5\}$ is in general position in $\Delta$. If \eqref{6.5} is violated, then the strict transform in $\tilde{\Delta}$ of the unique conic in $\Delta$  passing through these five points is contained in the base locus of $|D|_{\tilde{\Delta}}|$ with multiplicity $(m_2+m_3+m_4+m_5-d)+m_2+m_3+m_4+m_5-2(d-m_1)=2(m_1+\cdots+m_5)-3d>0$.
Assume $\Delta$ contains one of the transversal lines to $l_1$,  say $t_2$. 
We interpret this as a degeneration of the conic to the union of the lines spanned by $\{q_1,p_2\}$ and $\{q'_1,p_2\}$ while the plane $\Delta$ specializes.
Each of these two lines is contained in the base locus with multiplicity 
$2(m_1+\cdots+m_5)-3d$. This shows that 
the conic bundle over $\PP^1$, having fibers these conics, 
is contained in the base locus of $D$ with multiplicity $2(m_1+\cdots+m_5)-3d$.
 It is a surface of degree $\delta\ge3$,  
that is linearly equivalent to the following divisor 
$$
S=\delta H-\sum_{j=2}^5 E_i.
$$
Consider the divisor 
\begin{align*}
D'=D-S&=(d-\delta)H-m_1E_1-\sum_{j=2}^5(m_j-1)E_j\\
 & :=d'H-\sum_{j=1}^5m'_jE_j.
\end{align*}
Since $S\subseteq\textrm{Bs}(|D|)$, one obtains that $D$ is effective if and only if $D'$ is effective. 
Moreover $d\ge \delta$. Now, we distinguish two cases.

Case 1. If $m'_j=0$ for some $j\in\{2,\dots,5\}$, then we can think of $D'$ as a divisor living on the space blown-up along up to four lines and conclude by Theorem \ref{facets eff 4}.

Case 2. Assume $m'_j>0$ for all $j\in{1,\dots,5}$.
We claim that $(m'_1+\cdots+m'_5)-m'_j>d'$ for every $j$. This is very easy to check if $j\in\{2,\dots,5\}$. If $j=1$, we compute
\begin{align*}
0&<(2(m_1+\cdots+m_5)-3d)+(3\delta-8)=2(m'_1+\cdots+m'_5)-3d'\\
&=((m'_1+\cdots+m'_5)+m_1-2d')+((m'_1+\cdots+m'_5)-m_1-d')\\
& \le ((m'_1+\cdots+m'_5)-m_1-d'),
\end{align*}
where the last inequality follows from the fact that $D'$ is effective, 
as we already proved that \eqref{5.5} is necessary condition for this. 
Moreover the first inequality, that just follows from the assumption that \eqref{6.5} 
is violated and the fact that $\delta\ge3$, shows that $S$ is contained in the base 
locus of $|D'|$. Hence we reduced $D$ to a divisor $D'$ which satisfies the same 
assumptions and that has smaller degree and multiplicities along the lines. 
We can proceed as above until we reduce to either a divisor with negative degree
 or to Case 1. This concludes this part of the proof.

\medskip

We now prove that a divisor $D$ satisfying all inequalities is effective.

Consider first of all the case $m_{i_1}+m_{i_2}+m_{i_3}\le d$, for all triples of pairwise distinct indices $\{i_1,i_2,i_3\}\subset\{1,\dots,5\}$.
 Modulo reordering the indices, if necessary, we may assume $m_5\ge m_i$ for $i=1,\dots,4$.
Set $$T:=4H-\sum_{i=1}^5E_i-E_5.$$ This divisor is effective, because we can write it as a sum of effective divisors  $T=(2H-E_1-E_2-E_5)+(2H-E_3-E_4-E_5)$. 
Consider the following difference:
$$D'':=2D-m_5T.$$
If $D''$ is effective we conclude, in fact when describing the effective cone we are only interested in effectivity up to numerical equivalence.
 Notice that $D''$ is a divisors that lives in $\Bl_4\PP^3$, the blow-up of $\PP^3$ along four lines, in fact we can write
$$
D''=(2d-4m_5)H-\sum_{i=1}^4 (2m_i-m_5)E_i:=d''H-\sum_{i=1}^4m''_i E_i.
$$
Notice that by assumption, $d''\ge0$,  $m''_i\le d''$ and $m''_i+m''_j\le d''$. To show that $D''$ is effective, it is enough to verify that it satisfies the inequality \eqref{4.4} of Theorem \ref{facets eff 4}; we conclude by noticing that\\
$$
2\sum_{i=1}^4m''_i- 3d''=2(2(m_1+\cdots+m_5)-3d),
$$
so that the inequality  $2\sum_{i=1}^4m''_i\le 3d''$ is equivalent to \eqref{6.5}.

Modulo reordering the indices if necessary, assume that $m_1+m_2+m_3>d$, namely that the quadric through
the lines $l_1,l_2,l_3$ is contained in the base locus of $D$ with multiplicity 
$k_{123}=m_1+m_2+m_3- d$, see Lemma \ref{base locus for quadric}.
It is enough to prove that 
$$D-k_{123}(2H-E_1-E_2-E_3)$$ is effective. To conclude it is enough to check that  it satisfies \eqref{1.5}$,\dots,$\eqref{6.5} and then to use the preceding case.  
We leave the details to the reader.

\end{proof}

\subsection{Extremal rays of the effective cones}

\begin{corollary}\label{rays eff 2}
The  extremal rays of $\Eff(X_2)$ are 
\begin{align*}
\cone(E_i), & \quad  1\le i\le 2,\\
\cone(H-E_i), & \quad 1\le i\le 2.
\end{align*}
\end{corollary}
\begin{proof}
The proof of Theorem \ref{facets eff 3} implies that every effective divisor on $X_2$
can be written as linear combination with positive coefficients of 
the effective divisors $E_i$ and $H-E_i$, for $i=1,2$. Hence these divisors form a set of generators
for the effective cone. 

We now show that each of these generators spans an extremal rays 
of the effective cone in $\Nn^1(X_2)_\mathbb{R}\cong\mathbb{R}^3$.
In order to do this we use the equations
describing the facets of the effective cone in Theorem \ref{facets eff 3} and we obtain:
\begin{align*}
\cone(E_1)&=\{d=0\}\cap\{m_2=0\},
\\
\cone(E_2)&=\{d=0\}\cap\{m_1=0\},
\\
\cone(H-E_1)&=\{m_1=d\}\cap\{m_1+m_2=d\},
\\
\cone(H-E_2)&=\{m_2=d\}\cap\{m_1+m_2=d\}.
\end{align*}
\end{proof}

\begin{corollary}\label{rays eff 3 4 5}
The extremal rays of $\Eff(X_s)$, $s=3,4,5$ are 
\begin{align*}
\cone(E_i), & \quad  1\le i\le s,\\
\cone(H-E_i), & \quad 1\le i\le s,\\
\cone(2H-E_{i_1}-E_{1_2}-E_{1_3}), & \quad 1\le i_1<i_2<i_3\le s.
\end{align*}
\end{corollary}
\begin{proof}
The proofs of Theorem \ref{facets eff 3},  
Theorem \ref{facets eff 4} 
and Theorem \ref{facets eff 5} respectively provide a recipe for writing $D$ 
as sum of positive
 multiples of the generators listed.

To see that these generators span the extremal rays of the effective cone, we argue as in the proof
of Corollary \ref{rays eff 2}.
For $s=3$, using Theorem \ref{facets eff 3}, we obtain the following description: 
\begin{align*}
\cone(E_{i_1})&=\{d=0\}\cap\bigcap_{i\ne i_1}\{m_i=d\}\cap\{m_1+m_2+m_3-m_{i_1}=d\},
\\
\cone(H-E_{i_1})&=\{m_{i_1}=d\}\cap\bigcap_{i\ne i_1}\{m_{i_1}+m_i=d\}.
\end{align*}
Each ray of the form $\cone(E_{i_1})$ (resp. $\cone(H-E_{i_1})$) is intersection of
four (resp. three) hyperplanes of $\Nn^1(X_3)_\mathbb{R}\cong\mathbb{R}^4$.

For $s=4$, using Theorem \ref{facets eff 4}, we obtain 
\begin{align*}
\cone(E_{i_1})=&\{d=0\}\cap\bigcap_{i\ne i_1}\{m_i=d\}\cap\bigcap_{i,j\ne i_1}\{m_i+m_j=d\},
\\
\cone(H-E_{i_1})=&\{m_{i_1}=d\}\cap\bigcap_{i\ne i_1}\{m_{i_1}+m_i=d\}\\
& 
\cap\{(m_1+\cdots+m_4)+m_{i_1}=2d\},\\
\cone(2H-E_{i_1}-E_{i_2}-E_{i_3})=&\bigcap_{i,j\in\{i_1,i_2,i_3\}}\{m_i+m_j=d\}\\
& \cap
\bigcap_{i\in \{i_1,i_2,i_3\}}\{(m_1+\cdots+m_4)+m_i=2d\}\\
 & 
\cap\{2(m_1+\cdots+m_4)=3d\}.
\end{align*}

We leave it to the reader to verify the statement for $s=5$, using Theorem \ref{facets eff 5}.
\end{proof}


\begin{thebibliography}{99}

\bibitem{AlHi}
J.~Alexander, A.~Hirschowitz,
{\it Polynomial interpolation in several variables},
J. Algebraic Geom.  4 (1995), no. 2, 201--222.

\bibitem{BCHM}
C. Birkar, P. Cascini, C. D. Hacon, and J. McKernan,
{\it Existence of minimal models for varieties of log general type},
Journal of the American Mathematical Society 23 (2010), no. 2, 405--468.



\bibitem{BDP1}
 M.C.~Brambilla, O.~Dumitrescu and E.~Postinghel, {\it On a notion of speciality of linear systems in $\mathbb{P}^n$}, Trans. Am. Math. Soc. 367 (2015), 5447-5473

\bibitem{BDP3}
 M.C.~Brambilla, O.~Dumitrescu and E.~Postinghel, {\it On the effective cone of $\PP^n$ with $n+3$ points}, 
http://arxiv.org/pdf/1501.04094v2.pdf, to appear in Experimental Mathematics.

\bibitem{ale-hirsch} 
M.C.~Brambilla and G.~Ottaviani, {\it On the Alexander-Hirschowitz theorem},
J. Pure Appl. Algebra 212 (2008), no. 5, 1229--1251


\bibitem{CT} A.M.~Castravet and J.~Tevelev, {\it Hilbert's 14th problem and Cox rings}, Compos. Math. 142 (2006), no. 6, 1479--1498.


\bibitem{Ch1} 
K.~Chandler, {\it A brief proof of a maximal rank theorem for generic double points in projective space}, Trans. Amer. Math. Soc. 353 (2001), no. 5, 1907--1920.


\bibitem{Ch2}
K.~Chandler, {\it Linear systems of cubics singular at general points of projective space}, Compositio Mathematica 134 (2002), 269--282.

\bibitem{Chandler} 
K.~Chandler, {\it The geometric interpretation of Fr\"oberg-Iarrobino conjectures on infinitesimal neighbourhoods of points in projective space},  J. Algebra 286 (2005), no. 2, 421--455. 

\bibitem{Ciliberto} 
C.~Ciliberto, {\it Geometrical aspects of polynomial
  interpolation in more variables and
of Waring's problem}, European Congress of Mathematics, Vol. I 
(Barcelona, 2000), 289--316, Progr. Math., 201, Birkh\"auser, Basel, 2001. 

\bibitem{CLS}
D. A. Cox, J. B. Little and H. Schenck, {\it Toric Varieties}, Graduate Studies in Mathematics 124,
American Mathematical Society, Providence, RI, 2011.

\bibitem{DumPos}
O.~Dumitrescu and E.~Postinghel, {\it Vanishing theorems for linearly obstructed divisors}, 
arXiv:1403.6852 (2014).

\bibitem{DP}
O.~Dumitrescu and E.~Postinghel, {\it Positivity of divisors on blown up projective spaces}, 
 http://arxiv.org/pdf/1506.04726v3.pdf (2015).

\bibitem{ELMNP}
L. Ein, R. Lazarsfeld, M. Musta{\c{t}}{\u{a}}, M. Nakamaye, and M. Popa,
{\it Asymptotic invariants of base loci},  Ann. Inst. Fourier (Grenoble), vol. 56 (2006), no. 6, 1701--1734.

\bibitem{Froberg} R.~Fr\"oberg,
{\it An inequality for Hilbert series of graded algebras},
Math. Scand. 56 (1985), no. 2, 117--144.


\bibitem{Iarrobino}
A.~Iarrobino,
{\it Inverse system of symbolic power III. Thin algebras and fat points}, 
Compositio Math. 108 (1997), no. 3, 319--356. 

\bibitem{LazarsfeldI}
R.~Lazarsfeld.
{\it Positivity in algebraic geometry. I}, vol. 48 of
Ergebnisse der Mathematik und ihrer Grenzgebiete. 3. Folge. A Series of Modern Surveys in Mathematics [Results in Mathematics and Related Areas. 3rd Series. A Series of Modern Surveys in Mathematics].
Springer-Verlag, Berlin (2004)

\bibitem{MM}
S.~Mori and  S.~Mukai, {\it Classification of Fano 3-folds with $B_2\ge2$}, Manuscripta Math. 36 (1981/82), no. 2, 147--162. 


\bibitem{Po} 
 E.~Postinghel, {\it A new proof of the Alexander-Hirschowitz interpolation theorem},
Ann. Mat. Pura Appl. (4) 191 (2012), no. 1, 77��-94. 
\end{thebibliography}
\end{document}